\documentclass[12pt]{amsart}
\usepackage{etex}
\usepackage[all,cmtip]{xy} %package for drawing commutative diagrams
\usepackage{amssymb, amsmath, amscd, dcpic, pictexwd}
\usepackage{hyperref}
\usepackage{graphicx}
\usepackage{pdfpages}
\usepackage[overload]{empheq} %package for numbering subequations
\usepackage{color}
\usepackage{mathtools}

\numberwithin{equation}{section}
\newtheorem{thm}{Theorem}[section]
\newtheorem{question}[thm]{Question}
\newtheorem{cor}[thm]{Corollary}

\newtheorem{prop}[thm]{Proposition}

\newtheorem{exm}[thm]{Example}

\newtheorem{claim}[thm]{Claim}

\newcommand{\cH}{\mathcal H}
\newcommand{\cI}{\mathcal I}
\newcommand{\cJ}{\mathcal J}

\newcommand{\cL}{\mathcal L}
\newcommand{\cM}{\mathcal M}
\newcommand{\cN}{\mathcal N}
\newcommand{\cO}{\mathcal O}
\newcommand{\cP}{\mathcal P}

\newcommand{\cY}{\mathcal Y }

\newcommand{\bL}{\bf L}

\newcommand{\C}{\mathbb C}

\newcommand{\HH}{\mathbb{H}}
\renewcommand{\H}{\HH}

\newcommand{\PP}{\mathbb P}
\renewcommand{\P}{\PP}

\newcommand{\Z}{\mathbb Z}

\newcommand{\fg}{\mathfrak{g}}
\renewcommand{\gg}{\fg}

\newcommand{\onto}{\twoheadrightarrow}

\newcommand{\ra}{\rightarrow}

\newcommand{\bs}{\bigskip}

\newcommand{\End}{{\mbox {End~}}}
\newcommand{\Hom}{{\mbox{Hom}}}

\newcommand{\Aut}{{\mbox{Aut~}}}

\newcommand{\Res}{{\mbox{Res~}}}

\newcommand{\Der}{{\mbox{Der~}}}

\newcommand{\kerr}{{\mbox{ker~}}}
\renewcommand{\ker}{\kerr}

\def\comment#1{{}}
\def\question#1{{}}

\newcommand{\quash}[1]{}  
\newcommand{\red}[1]{}  
\newcommand{\blue}[1]{{}}
\newcommand{\green}[1]{{}}

\newcommand{\Cx}{\mathbb{C}^\times}

\title{Holonomic Systems for Period Mappings}
\author{Jingyue Chen, An Huang, Bong H. Lian}

\begin{document}

\maketitle
{\it Dedicated to Professor Shing-Tung Yau on the occasion of his 65th birthday.}

\begin{abstract}
Period mappings were introduced in the sixties  \cite{G} to study variation of complex structures of families of algebraic varieties.
The theory of tautological systems was introduced recently \cite{LSY}\cite{LY} to understand period integrals of algebraic manifolds. In this paper, we give an explicit construction of a tautological system for each component of a period mapping. 
\end{abstract}

\tableofcontents
%\baselineskip=16pt plus 1pt minus 1pt
%\parskip=\baselineskip

\pagenumbering{arabic}
\addtocounter{page}{0}

\section{Setup}
\subsection{Tautological systems}
We will follow notations in \cite{HLZ}.
Let $G$ be a connected algebraic group over $\C$. Let $X$ be a complex projective $G$-variety and let $\cL$ be a very ample $G$-bundle over $X$ which gives rise to a $G$-equivariant embedding 
\[X\ra\P(V),\]
where $V=\Gamma(X,\cL)^\vee$. Let $n=\dim V$. We assume that the action of $G$ on $X$ is locally effective, i.e. $\ker(G\ra\Aut (X))$ is finite. 
%Let $\Cx$ be the multiplicative group acting on $V$ by homotheties. 
Let $\hat{G}:=G\times\Cx$, whose Lie algebra is $\hat{\gg}=\gg\oplus \C e$, where $e$ acts on $V$ by identity. We denote by $Z:\hat{G}\ra \text{GL}(V)$ the corresponding group representation, and by $Z:\hat{\gg}\ra\End(V)$ the corresponding Lie algebra representation. Note that under our assumption, $Z:\hat{\gg}\ra\End(V)$ is injective.

Let $\hat{\iota}: \hat{X}\subset V $ be the cone of $X$, defined by the ideal $I(\hat{X})$. Let $\beta:\hat{\gg}\ra \C$ be a Lie algebra homomorphism. Then a {\it tautological system } as defined in \cite{LSY}\cite{LY} is the cyclic D-module on $V^\vee$
\[\tau(G,X,\cL,\beta)=D_{V^\vee}/D_{V^\vee}J(\hat{X})+D_{V^\vee}(Z(x)+\beta(x), x\in\hat{\gg}),\]
where $$J(\hat{X})=\{\widehat{P}\mid P\in I(\hat{X})\}$$
is the ideal of the commutative subalgebra $\C[\partial]\subset D_{V^\vee}$ obtained by the Fourier transform of $I(\hat{X})$. Here $\widehat{P}$ denotes the Fourier transform of $P$.

Given a basis $\{a_1,\ldots, a_n\}$ of $V$, we have $Z(x)=\sum_{ij}x_{ij}a_i\frac{\partial}{\partial{a_j}}$, where $(x_{ij})$ is the matrix representing $x$ in the basis. Since the $a_i$ are also linear coordinates on $V^\vee$, we can view $Z(x)\in\Der \C[V^\vee]\subset D_{V^\vee}$. In particular, the identity operator $Z(e)\in\End V$ becomes the Euler vector field on $V^\vee$.

Let $X$ be a $d$-dimensional compact complex manifold such that its anti-canonical line bundle $\cL:=\omega_X^{-1}$ is very ample. We shall regard the basis elements $a_i$ of $V=\Gamma(X,\omega_X^{-1})^{\vee}$ as linear coordinates on $V^{\vee}$. Let $B:=\Gamma(X,\omega_X^{-1})_{sm}\subset V^{\vee}
$, which is Zariski open.

 Let $\pi:\cY\ra B%\subset V^{\vee}
$
 be the family of smooth CY hyperplane sections $Y_a\subset X$, and let $\H^{\text{top}}$ be the Hodge bundle over $B$ whose fiber at $a\in B$ is the line $\Gamma(Y_a,\omega_{Y_a})\subset H^{d-1}(Y_a)$. In \cite{LY} the period integrals of this family are constructed by giving a canonical trivialization of $\H^{\text{top}}$. Let $\Pi$ be the period sheaf of this family, i.e. the locally constant sheaf generated by the period integrals.

%Let $V=\Gamma(X,\omega_X^{-1})^\vee$ and $\cL=\omega_X^{-1}$.  Let 
 Let $G$ be a connected algebraic group acting on $X$.

\begin{thm}[See \cite{LY}]\label{tautological}
The period integrals of the family $\pi:\cY\ra B$ are solutions to 
\[\tau\equiv\tau(G,X,\omega_X^{-1},\beta_0)\]
where $\beta_0$ is the Lie algebra homomorphism with $\beta_0(\gg)=0$ and $\beta_0(e)=1$.
\end{thm}

In \cite{LSY} and \cite{LY}, it is shown that if $G$ acts on $X$ by finitely many orbits, then 
$\tau$ is regular holonomic.

\bs
\subsection{Period mapping} 

Now we consider the period mapping of the family $\cY$: 
\begin{align*}
\cP^{p,k}: B&\ra\text{Gr}(b^{p,k},H^k(Y_{a_0},\C))/\Gamma\\
a&\mapsto [F^pH^k(Y_a,\C)]\subset [H^k(Y_a,\C)]\cong [H^k(Y_{a_0},\C)]
\end{align*}
where  $a_0\in B$ is a fixed base point, $$b^{p,k}:=\dim F^pH^k(Y_a,\C)=\dim F^pH^k(Y_{a_0},\C),$$
and $\Gamma$ is the monodromy group acting on $\text{Gr}(b^{p,k},H^k(Y_{a_0},\C))$.

Consider the local system $R^k\pi_*\C$, its stalk at $a\in B$ is 
$H^k(Y_a,\C)$. The Gauss-Manin connection on the vector bundle $\cH^k=R^k\pi_*\C\otimes \cO_B$ %associated to the local system $R^k\pi_*\C$
 has the property that 
$$\nabla F^p\cH^k\subset F^{p-1}\cH^k\otimes\Omega_B.$$
%(See \cite[I, p240]{Vo}.)

If we choose a vector field $v\in\Gamma(B,TB)$, then $\nabla_v F^p\cH^k\subset F^{p-1}\cH^k$, i.e. $\nabla_v F^pH^k(Y_a,\C)\subset F^{p-1}H^k(Y_a,\C).$

%It is known that the Gauss-Manin connection has the property that  $$\nabla F^pH^k(Y_a,\C)\subset F^{p-1}H^k(Y_a,\C).$$ 
\blue{In Voisin it is $\nabla F^p\cH^k\subset F^{p-1}\cH^k\otimes\Omega_B$, do we need to worry about $\Omega_B$?\\} \red{It is better just to write the sheaf version in Voisin. But to drop $\Omega_B$, you must replace $\nabla$ by $\nabla_v$ for $v\in\Gamma(B,TB)$.}  \blue{OK.}

\bs

{\it Throughout this paper we shall consider the case $k=d-1.$}

$\H^{\text{top}}$ is the bundle on $B$ whose fiber is $$H^0(Y_a,\Omega^{d-1})=H^{d-1,0}(Y_a)=F^{d-1}H^{d-1}(Y_a,\C).$$
Thus $\H^{\text{top}}=F^{d-1}\cH^{d-1}.$
Theorem \ref{tautological} tells us that its integral over a $(d-1)$-cycle on each fiber $Y_{a_0}$ is governed by the tautological system $\tau$.

We shall describe the Gauss-Manin connection explicitly below.

%$\omega\in H^{d-1}(Y_a,\C)$. 
\green{I think this last sentence should be expressed more concretely: e.g. according to HLZ, for our canonical choice of trivilization of $H^{\top}$ explicitly as a global Poincare residue formula, the Gauss-Manin connection takes this specific form as taking partial derivatives w.r.t. $a_i$, on our chosen trivilization: what this taking partial derivative means could be explicitly written as a formula as in HLZ, just as what one would formally expect. Maybe more precisely is to say that this is how Gauss-Manin D-module action on the open complement family can be described explicitly. And so perhaps the description of the trivilization may be written before this paragraph.}

 \red{Show spell out that $f=\sum a_i a_i^*$ is thought of as the universal section $V^\vee\times X\ra\bL$. Also mention explicitly in the Intro. that we will mostly follow notations in [HLZ].} \blue{OK.}

Let $(K_X^{-1})^\times$ denote the complement of the zero section in the total space $K_X^{-1}$ of $\cL$.
Consider the principal $\Cx$-bundle $(K_X^{-1})^\times\ra X$ (with right action). Then there is a natural one-to-one correspondence between sections of $\omega_X^{-1}$ and $\Cx$-equivariant morphisms $f: (K_X^{-1})^\times\ra \C$, i.e. $f(m\cdot h^{-1})=hf(m)$. We shall write $f_a$ the function that represents the section $a$. Since $(K_X^{-1})^\times$ is a CY bundle over $X$ by \cite{LY}, it admits a global non-vanishing top form $\hat{\Omega}$. Let $x_0$ be the vector field generated by $1\in \C=\text{Lie} (\Cx)$. Then $\Omega:=\iota_{x_0}\hat{\Omega}$ is a $G$-invariant $\Cx$-horizontal form of degree $d$ on $(K_X^{-1})^\times$. 
 Moreover, since $\frac{\Omega}{f_a}$
is $G\times \Cx$-invariant, it defines a family of meromorphic top form on $X$ with pole along $V(f_a)$ \cite[Thm. 6.3]{LY}.

 Let $a_i^*\in \Gamma(X, \omega_X^{-1})$ be the dual basis of $a_i$. Let $f=\sum a_i a_i^*$ be the universal section of $V^\vee\times X\ra K_X^{-1}$. Let $V(f)$ be the universal family of hyperplane sections,  where $V(f_a)=Y_a$ is the zero locus of the section $f_a\equiv a\in V^\vee$. 
 Let $U:=V^\vee\times X-V(f)$ and $U_a=X-V(f_a)$.  Let $\pi^\vee: U\ra V^\vee$ denote the projection.

 Let $\hat{\pi}^\vee: U\ra B$ be the restriction of  $\pi^\vee$ to $B$.
Then there is a vector bundle $\hat{\cH}^d:=R^d(\hat{\pi}^\vee)_*\C\otimes\cO_B$ whose fiber is $H^d(X-V(f_a))$ at $a\in B$. Then $\frac{\Omega}{f}$ is a global section of $F^d\hat{\cH}^d$. And the Gauss-Manin connection on $\frac{\Omega}{f}$ is $$\nabla_{\partial a_i}\frac{\Omega}{f}=\frac{\partial}{\partial a_i}\frac{\Omega}{f},$$ where $\partial_{a_i}:=\frac{\partial}{\partial a_i}, i=1,\ldots,n.$

Consider the residue map $\Res: H^d(X-V(f_a))\ra H^{d-1}(Y_a,\C)$, it is shown in \cite{LY} that $\Res\frac{\Omega}{f}\in  \Gamma(B, \cH^{d-1})$ is a canonical global trivialization of $\H^{\text{top}}=F^{d-1}\cH^{d-1}.$ And similarly we have $$\nabla_{\partial a_i}\Res\frac{\Omega}{f}=\frac{\partial}{\partial a_i}\Res\frac{\Omega}{f}=\Res\frac{\partial}{\partial a_i}\frac{\Omega}{f}.$$

In \cite[Coro. 2.2 and Lemma 2.6]{HLZ}, it is shown that 
\begin{thm}\label{tau-surj}
If $\beta(\gg)=0$ and $\beta(e)=1$, there is a canonical surjective map 
\[\tau\onto H^0\pi^\vee_+\cO_U, \quad 1\mapsto \dfrac{\Omega}{f}.\]
\end{thm}

We now want  to give an explicit description of each step of the Hodge filtration.

%say sth about sufficiently ample

Let $X$ be a projective variety of dimension $d$ and $Y$ a smooth hypersurface. We make the following hypothesis:

\begin{flalign*}\label{suffi-ample}\tag{*}
\text{ For every } i>0, \, k>0,\,& j\geq0, \text{ we}\text{ have}& \\
H^i(X,&\Omega^j_X(kY))=0,  \\
\text{where }  \Omega^j_X(kY)=\Omega^j_X\otimes \cO_X&(Y)^{\otimes k}.&
\end{flalign*}
\begin{thm}[Griffiths]\label{Griffiths}Let $X$ be a projective variety and $Y$ a smooth hypersurface. Assume \eqref{suffi-ample} holds. %$\cO_X(Y)$ is very ample. 
Then for every integer $p$ between $1$ and $d$, the image of the natural map 
$$H^0(X,\Omega_X^d(pY))\ra H^d(X-Y, \C)$$
which to a section $\alpha$ (viewed as a meromorphic form on $X$ of degree $d$, holomorphic on $X-Y$ and having a pole of order less than or equal to $p$ along $Y$) associates its de Rham cohomology class, is equal to $F^{d-p+1}H^d(X-Y)$ (See \cite[II, p 160]{Vo}.)
\end{thm}

\begin{cor}\label{filtration}
Assume \eqref{suffi-ample} holds for smooth CY hypersurfaces $Y_a\subset X$. %$H^i(X,\Omega^j_X(kY_a))=0$ for all $i>0,  k>0$ and $j\geq0$.
 Then the de Rham classes of $$\{\nabla_{\partial{a_{\iota_1}}}\cdots \nabla_{\partial{a_{\iota_{p-1}}}}\frac{\Omega}{f_a}\}_{1\leq \iota_1,\ldots, \iota_{p-1}
\leq n}$$ generate the filtration $F^{d-p+1}H^d(U_a)$ for $1\leq p\leq d$.
\end{cor}
\red{Typos in subscripts of $\nabla_{a_{\iota_1}}$.} \blue{Fixed.}
\begin{proof}
By our assumption $X$ is  projective and $\cO_X(Y_a)=\omega_X^{-1}$ is very ample.
By Theorem \ref{Griffiths} it is sufficient to show that 
\begin{equation}\label{form}
\C\{\nabla_{\partial{a_{\iota_1}}}\cdots \nabla_{\partial{a_{\iota_{p-1}}}}\frac{\Omega}{f_a}\}_{1\leq \iota_1,\ldots, \iota_{p-1}
\leq n}=H^0(X,\Omega_X^d(pY_a))
\end{equation}

Since $Y_a$ are CY hypersurfaces, there is an isomorphism
$$\cO_X\simeq\Omega^d_X(Y_a), \quad 1\leftrightarrow \frac{\Omega}{f_a}. $$ Let $\cM(pY_a)$ be the sheaf of meromorphic functions  \green{sheaf of meromorphic functions...}\blue{Fixed.}  with a pole along $Y_a$ of order less than or equal to $p$. Then there is an isomorphism
$$\cO_X(Y_a)^{p-1}\simeq \cM((p-1)Y_a), \quad g\leftrightarrow \frac{g}{f_a^{p-1}}.$$ 
Thus 
we have an isomorphism
\[\cO_X \otimes \cO_X(Y_a)^{p-1}\simeq\Omega_X^d(Y_a)\times \cM((p-1)Y_a), \, 1\otimes g\leftrightarrow \frac{\Omega}{f_a}\frac{g}{f_a^{p-1}}.\]

Since $$\Omega^d_X(pY_a):=\Omega^d_X\otimes \cO_X(Y_a)^p\equiv \cO_X(Y_a)^{p-1},$$
we have
\[H^0(X,\Omega_X^d(pY_a))=\{g\frac{\Omega}{f_a^{p}}\mid g\in H^0(X, \cO_X(Y_a)^{p-1}).\}\]

\bs
For $p=1$, %since $\frac{\Omega}{f_a}$ generate all holomorphic top form, 
 the statement is clearly true. 
 
 For $p=2$, $f=\sum_i a_i a_i^*$, $\nabla_{\partial{a_i}}\frac{\Omega}{f_a}=\frac{\partial}{\partial a_i}\frac{\Omega}{f_a}=-\frac{a_i^*\Omega}{f_a^2}.$
Since $\{\frac{a_i^*\Omega}{f_a^2}\mid a_i^*\in  H^0(X,\omega_X^{-1})\}=H^0(X,\Omega_X^d(2Y_a))$, \eqref{form} is true. 
\begin{claim}
For a very ample line bundle $L$ over $X$, 
\[H^0(X,L^{ k})\otimes H^0(X,L^{ l})\ra H^0(X,L^{{k+l}})%, \quad a_i^*\otimes a_j^*\mapsto a_i^*a_j^*
\]
is surjective.
\end{claim}
\begin{proof}[Proof of claim]
Since $L$ is very ample, let $V:=H^0(X,L)^\vee$, it follows that $L=\cO_{\P V}(1)|_X$. Since restriction commutes with tensor product, $L^k=\cO_{\P V}(k)|_X$. And since $H^0(\P V,\cO_{\P V}(k))\otimes H^0(\P V,\cO_{\P V}(l))\ra H^0(\P V,\cO_{\P V}(k+l))$ is surjective, it follows that $H^0(X,L^k)\otimes H^0(X,L^l)\ra H^0(X,L^{k+l})$ is surjective.  
\end{proof}

For $p=3$, $\nabla_{\partial{a_i}}\nabla_{\partial{a_j}}\frac{\Omega}{f_a}=\frac{2a_i^*a_j^*\Omega}{f_a^3}$. 
Since $\cL:=\omega_X^{-1}$ is very ample, $H^0(X,\cL)\otimes H^0(X,\cL)\ra H^0(X,\cL^2)$ is onto. So $\{a_i^*a_j^*\}_{1\leq i,j \leq n}$ generate $H^0(X,\cL^2)$.
Thus  $$\C\{\frac{a_i^*a_j^*\Omega}{f_a^3}\}_{1\leq i,j
\leq n}
=H^0(X,\Omega_X^d(3Y_a)),$$ and therefore \eqref{form} holds.

Similarly, since $H^0(X,\cL^{p-2})\otimes H^0(X,\cL)\ra H^0(X,\cL^{p-1})$ is surjective, by induction $a_{\iota_1}^*\cdots a_{\iota_{p-1}}^*$ generate $H^0(X,\cL^{p-1})$.
Therefore the de Rham classes of 
$$\left\{ \nabla_{\partial_{a_{\iota_1}}}\cdots \nabla_{\partial_{a_{\iota_{p-1}}}}\frac{\Omega}{f_a}=(-1)^{p-1}\frac{(p-1)!a_{\iota_1}^*\cdots a_{\iota_{p-1}}^*\Omega}{f_a^p}\right\}_{1\leq \iota_1,\ldots, \iota_{p-1}
\leq n}
$$
\red{Typos in subscripts $\nabla_{a_{\iota_1}}$?} \blue{fixed.}
generate $F^{d-p+1}H^d(U_a).$
\end{proof}

\bs

\begin{cor}
Assume \eqref{suffi-ample} holds for smooth CY hypersurfaces $Y_a\subset X$. Then the de Rham classes of $$\{\nabla_{\partial_{a_{\iota_1}}}\cdots \nabla_{\partial_{a_{\iota_{p-1}}}}\Res\frac{\Omega}{f_a}\}_{1\leq \iota_1,\ldots, \iota_{p-1}\leq n}$$ generate the filtration $F^{d-p}H^{d-1}_{\text{van}}(Y_a)$ for $1\leq p\leq d$.
\end{cor}

\begin{proof}
Consider the exact sequence
\[0\ra H^d_{\text{prim}}(X)\ra H^d(X-V(f_a))\xrightarrow{\text{Res}} H^{d-1}_{\text{van}}(Y_a,\C)\ra 0,\]
%The  residue map is  surjective and preserves the Hodge filtration. 
it shows that $H^d(X-V(f_a))$ is mapped surjectively onto the vanishing cohomology of $H^{d-1}(Y_a,\C)$ under the residue map. Since the residue map preserves the Hodge filtration, by Corollary \ref{filtration} the result follows.
 \end{proof}

The goal of this paper is to construct a regular holonomic differential system that governs the $p$-th derivative of period integrals for each $p\in\Z.$ By the preceding corollary, this provides a differential system for ``each step'' of the period mapping of the family $\cY$.

\section{Scalar system for first derivative}

 We shall use $\tau$ to denote interchangeably both the D-module and its left defining ideal. Let $P(\zeta)\in I(\hat{X})$ where $\zeta\in V$, then its Fourier transform $\widehat{P}=P(\partial_a), \,a\in V^\vee$. Then the tautological system $\tau=\tau(G,X,\omega_X^{-1},\beta_0)$ for $\Pi_\gamma(a)$ becomes the following system of differential equations:

\begin{subequations}\label{per}
    \begin{align}[left = \empheqlbrace\,]
      & P(\partial_a)\phi(a)=0\label{per1}\\
      &(\sum_{i,j} x_{ij}a_i\frac{\partial}{\partial a_j})\phi(a)=0\label{per2}\\
      &(\sum_i a_i\frac{\partial}{\partial a_i}+1)\phi(a)=0.\label{per3}
    \end{align}
    \end{subequations}

 Let $\omega_a:=\Res \frac{\Omega}{f_a}$ for $a\in B$. Since the topology of $Y_a$ doesn't change, we choose a $(d-1)$-cycle in $H^{d-1}(Y_{a_0},\C)$ for some $a_0\in B$. Then the period integral becomes $\Pi_\gamma(a)=\int_{\gamma}\omega_a$. 
Then by Theorem \ref{tautological}, $\Pi_\gamma(a)$ are solutions of $\tau$. By \cite{BHLSY} and \cite{HLZ}, if $X$ is a projective homogeneous space, then $\tau$ is complete, meaning that the solution sheaf agrees with the period sheaf.
\blue{Is it clear?}

\bs
The goal of this section is to write a system of scalar valued partial differential equations whose solution contains all the information of  first order partial derivatives of period integrals.
%whose solution will give information of all first partial derivatives of the period integrals.

\subsection{Vector valued system.} Taking derivatives of equations in $\tau$ gives us a vector valued system of differential equations that involve all first order derivatives of period integrals.

\bs
Let $\phi_k(a):=\frac{\partial}{\partial a_k}\phi(a),\, 1\leq k\leq n.$

From equation \eqref{per1} we have
\[\frac{\partial}{\partial a_k}p(\partial_a)\phi(a)=P(\partial_a)\frac{\partial}{\partial a_k}\phi(a)=P(\partial_a)\phi_k(a)=0.\]
From equation \eqref{per2} we have
\begin{align*}
\frac{\partial}{\partial a_k}(\sum_{i,j} x_{ij}a_i\frac{\partial}{\partial a_j})\phi(a)%&=(\sum_{i,j} x_{ij}\frac{\partial}{\partial a_k}a_i\frac{\partial}{\partial a_j})f(a)\\
&=(\sum_{i,j} x_{ij}a_i\frac{\partial}{\partial a_j}\frac{\partial}{\partial a_k})\phi(a)+(\sum_j x_{kj}\frac{\partial}{\partial a_j})\phi(a)\\
&=(\sum_{i,j} x_{ij}a_i\frac{\partial}{\partial a_j})\phi_k(a)+\sum_j x_{kj}\phi_j(a)
=0.
\end{align*}
From equation \eqref{per3} we have
\begin{align*}
\frac{\partial}{\partial a_k}(\sum_i a_i\frac{\partial}{\partial a_i}+1)\phi(a)&=(\sum_i a_i\frac{\partial}{\partial a_i}\frac{\partial}{\partial a_k}+\frac{\partial}{\partial a_k}+\frac{\partial}{\partial a_k})\phi(a)\\
&=(\sum_i a_i\frac{\partial}{\partial a_i}+2)\phi_k(a)=0
\end{align*}
We also have a relation between derivatives
$\frac{\partial}{\partial a_i}\frac{\partial}{\partial a_j}\phi(a)=\frac{\partial}{\partial a_j}\frac{\partial}{\partial a_i}\phi(a),$
which implies $$\frac{\partial}{\partial a_i}\phi_j(a)=\frac{\partial}{\partial a_j}\phi_i(a).$$

Then we get a system of differential equations whose solutions are vector valued of the form $(\phi_1(a),\ldots ,\phi_n(a))$
%where $n:=\dim V$ 
as follows:
\begin{subequations}\label{devec}
    \begin{align}[left = \empheqlbrace\,]
      & P(\partial_a)\phi_k(a)=0,\; \forall 1\leq k\leq n\label{devec1}\\
      & %\sum_{i,j} x_{ij}a_if_j(a)\label{devec2}=0\\
      (\sum_{i,j} x_{ij}a_i\frac{\partial}{\partial a_j})\phi_k(a)+\sum_j x_{kj}\phi_j(a)=0, \;\forall 1\leq k\leq n\label{devec2}\\
      &(\sum_i a_i\frac{\partial}{\partial a_i}+2)\phi_k(a)=0,\; \forall 1\leq k\leq n\label{devec3}\\
      & \frac{\partial}{\partial a_i}\phi_j(a)=\frac{\partial}{\partial a_j}\phi_i(a),\;\forall 1\leq i,j\leq n\label{devec4}.
    \end{align}
    \end{subequations}

Then by Theorem \ref{tautological}, $\left(\frac{\partial}{\partial a_1} \Pi_\gamma(a),\ldots,\frac{\partial}{\partial a_n}\Pi_\gamma(a)\right)$ are solutions to system \eqref{devec}.

\subsection{Scalar valued system}
Now let $b_1,\ldots,b_n$ be another copy of the basis of $V^\vee$. We now construct, by an elementary way, a system of differential equations over $V^\vee\times V^\vee$ that is equivalent to \eqref{devec}, but whose solutions are function germs on $V^\vee\times V^\vee$. Consider the system %Let $f(a,b):=\sum_k b_kf_k(a)$  
 %Consider the following system of differential equations:
\begin{subequations}\label{desca}
    \begin{align}[left = \empheqlbrace\,]
      & P(\partial_a)\phi(a,b)=0\label{desca1}\\
      &(\sum_{i,j} x_{ij}a_i\frac{\partial}{\partial a_j}+\sum_{i,j} x_{ij}b_i\frac{\partial}{\partial b_j})\phi(a,b)=0\label{desca2}\\
      &(\sum_i a_i\frac{\partial}{\partial a_i}+2)\phi(a,b)=0\label{desca3}\\
      & \frac{\partial}{\partial a_i}\frac{\partial}{\partial b_j}\phi(a,b)=\frac{\partial}{\partial a_j}\frac{\partial}{\partial b_i}\phi(a,b),\;\forall 1\leq i,j\leq n\label{desca4}\\
      &\frac{\partial}{\partial b_i}\frac{\partial}{\partial b_j}\phi(a,b)=0\label{desca5},\;\forall 1\leq i,j\leq n\\
      &(\sum_i b_i\frac{\partial}{\partial b_i}-1)\phi(a,b)=0\label{desca6}.
    \end{align}
\end{subequations}

\begin{thm}
By setting $\phi(a,b)=\sum_k b_k \phi_k(a)$ and $\phi_k(a)=\frac{\partial}{\partial b_k}\phi(a,b)$, the systems \eqref{devec} and \eqref{desca} are equivalent.
\end{thm}

\begin{proof}
First we show that if $(\phi_1(a),\ldots ,\phi_n(a))$ is a solution to system \eqref{devec}, let $\phi(a,b)=\sum_k b_k\phi_k(a)$, then $\phi(a,b)$ is a solution to system \eqref{desca}.

Since $\phi(a,b)=\sum_k b_k\phi_k(a)$, equation \eqref{devec1} implies that 
\[ P(\partial_a)\phi(a,b)= P(\partial_a)\sum_k b_k\phi_k(a)=\sum_k b_kP(\partial_a)\phi_k(a)=0,\]
thus equation \eqref{desca1} holds.

Equation \eqref{desca2} can be shown as follows:
\begin{align*}
&(\sum_{i,j} x_{ij}a_i\frac{\partial}{\partial a_j}+\sum_{i,j} x_{ij}b_i\frac{\partial}{\partial b_j})(\sum_k b_k\phi_k(a))\\
&=\sum_k b_k((\sum_{i,j} x_{ij}a_i\frac{\partial}{\partial a_j})\phi_k(a))+\sum_{i,j} x_{ij}b_i\phi_j(a)\\
&\overset{\eqref{devec2}}{=}\sum_k b_k(-\sum_j x_{kj}\phi_j(a))+\sum_{i,j} x_{ij}b_i\phi_j(a)=0.
\end{align*}

Equation \eqref{devec3} shows that $\phi_k(a)$ is homogeneous of degree $-2$ in $a$, it implies that $\phi(a,b)$ is also homogeneous of degree $-2$ in $a$, which implies equation \eqref{desca3}.

Since $\frac{\partial}{\partial b_k}\phi(a,b)=\phi_k(a)$, equation \eqref{devec4} is the same as saying that 
\[\frac{\partial}{\partial a_i}\frac{\partial}{\partial b_j}\phi(a,b)=\frac{\partial}{\partial a_j}\frac{\partial}{\partial b_i}\phi(a,b)\quad\forall i,j,\]
which is equation \eqref{desca4}.

Since $\phi(a,b)$ is linear in $b$, equation \eqref{desca5} holds. $\phi(a,b)$ is homogeneous of degree $1$ in $b$, it implies equation \eqref{desca6}.

%The only nontrivial step is to show that equation \ref{desca2} can be derived from system \ref{devec}.

\bs

Next we show that if $\phi(a,b)$ is a solution to system \eqref{desca}, set $\phi_k(a)=\frac{\partial}{\partial b_k}\phi(a,b)$, then $(\phi_1(a),\ldots ,\phi_n(a))$ is a solution to system \eqref{devec}.

Equation \eqref{desca5} tells us that $\phi(a,b)$ is linear in $b$, i.e. there exists functions $h_k(a)$ and $g(a)$ on $V^\vee$ such that
$$\phi(a,b)=\sum_k b_kh_k(a)+g(a).$$
Equation \eqref{desca6} shows that $\phi(a,b)$ is homogeneous of degree $1$ in $b$, which implies that $g=0$ and
$$\phi(a,b)=\sum_k b_kh_k(a).$$
Now $\phi_k(a)=\frac{\partial}{\partial b_k}\phi(a,b)=\frac{\partial}{\partial b_k}(\sum_k b_kh_k(a))=h_k(a)$ and we thus can write $$\phi(a,b)=\sum_k b_k\phi_k(a).$$

Equation \eqref{desca1} shows that $$P(\partial_a)\sum_k b_k\phi_k(a)=\sum_k b_k (P(\partial_a)\phi_k(a))=0.$$
Since the $b_i$'s are linearly independent, this further shows that 
$$P(\partial_a)\phi_k(a)=0\quad \forall k,$$
which coincides with equation \eqref{devec1}.

From equation \eqref{desca2} we can see that
\begin{align*}
&(\sum_{i,j} x_{ij}a_i\frac{\partial}{\partial a_j}+\sum_{i,j} x_{ij}b_i\frac{\partial}{\partial b_j})(\sum_k b_k\phi_k(a))\\
&=\sum_k b_k(\sum_{i,j} x_{ij}a_i\frac{\partial}{\partial a_j})\phi_k(a)+\sum_{i,j} x_{ij}b_i\phi_j(a)\\
&=\sum_k b_k((\sum_{i,j} x_{ij}a_i\frac{\partial}{\partial a_j})\phi_k(a)+\sum_jx_{kj}\phi_j(a))=0.
\end{align*}
Since the $b_i$'s are linearly independent, we have $$(\sum_{i,j} x_{ij}a_i\frac{\partial}{\partial a_j})\phi_k(a)+\sum_jx_{kj}\phi_j(a)=0, \quad\forall k,$$
which is equation \eqref{devec2}.

Equation \eqref{desca3} shows that $\sum_k b_k\phi_k(a)$ is homogeneous of degree $-2$ in $a$, which implies that $\phi_k(a)$ is homogeneous of degree $-2$ in $a$ as well, which would imply equation \eqref{devec3}.

It's also clear that equation \eqref{desca4} implies \eqref{devec4}.

Therefore the two systems are equivalent in the above sense.
\end{proof}

Therefore $\sum_k b_k\frac{\partial}{\partial a_k}\Pi_\gamma(a)$ are solutions to system \eqref{desca}.

\blue{Seems like we don't use equations \eqref{devec4} and \eqref{desca4} in this proof?} \red{not sure what your mean. You had a statement involving these two equations at the end of the proof.} \blue{We do need these equations.}

\subsection{Regular holonomicity of the new system}
In this section we will show that system \eqref{desca} is regular holonomic, by extending the proof for the original tautological system in paper \cite{LSY}.

\bs

Let
$\cM:=D_{V^\vee\times V^\vee}/\cJ$ where $\cJ$ is the left ideal generated by the operators in system \eqref{desca}.

\begin{thm}\label{holonomicity}
Assume that the $G$-variety $X$ has only a finite number of $G$-orbits. Then the D-module $\cM$ is regular holonomic.
\end{thm}

\begin{proof}
Consider the Fourier transform:
\[
\widehat{a_i}=\frac{\partial}{\partial \zeta_i},\;
\widehat{b_i}=\frac{\partial}{\partial \xi_i},\;
\widehat{\frac{\partial}{\partial a_i}}=-\zeta_i,\;
\widehat{\frac{\partial}{\partial b_i}}=-\xi_i.\;
\]
where $\zeta,\xi\in V$.
The Fourier transform of the D-module $\cM=D_{V^\vee\times V^\vee}/\cJ$ is $\widehat{\cM}=D_{V\times V}/\widehat{\cJ}$, where $\widehat{\cJ}$ is the $D_{V\times V}$-ideal generated by the following operators:
\begin{subequations}\label{desca-f}
    \begin{align}[left = \empheqlbrace\,]
      & P(\zeta)\label{desca-f1}\\
      &\zeta_i\xi_j-\zeta_j\xi_i,\;\forall 1\leq i,j\leq n\label{desca-f4}\\
      &\xi_i\xi_j\label{desca-f5},\;\forall 1\leq i,j\leq n\\
      &\sum_{i,j} x_{ij}\frac{\partial}{\partial \zeta_i}\zeta_j+\sum_{i,j} x_{ij}\frac{\partial}{\partial \xi_i}\xi_j=\sum_{i,j} x_{ji}\zeta_i\frac{\partial}{\partial \zeta_j}+\sum_{i,j} x_{ji}\xi_i\frac{\partial}{\partial \xi_j}+\sum 2x_{ii}\label{desca-f2}\\
      &\sum_i \frac{\partial}{\partial \zeta_i}(-\zeta_i)+2=\sum_i -\zeta_i\frac{\partial}{\partial \zeta_i}-n+2\label{desca-f3}\\
      &\sum_i \frac{\partial}{\partial \xi_i}(-\xi_i)-1=\sum_i -\xi_i\frac{\partial}{\partial \xi_i}-n-1\label{desca-f6}.
    \end{align}
\end{subequations}
Consider the $G\times\Cx\times\Cx$-action on $V\times V$ where $G$ acts diagonally and each $\Cx$ acts on $V$ by scaling. Consider the ideal $\cI$ generated by \eqref{desca-f1}, \eqref{desca-f4} and \eqref{desca-f5}. Operator \eqref{desca-f5} tells us that $\xi_i=0$ for all $i$, thus $(V\times V)/\cI\subset V\times \{0\}$. The ideal $\cI$ also contains \eqref{desca-f1}, thus $(V\times V)/\cI=\hat{X}\times \{0\}$, which is an algebraic variety. Operator \eqref{desca-f2} comes from the $G$-action on $V\times V$, operators \eqref{desca-f3} and \eqref{desca-f6} come from each copy of $\Cx$-action on $V$. Now from our assumption %({\color{blue}need to have good description of assumptions}) 
the $G$-action on $X$ has only a finite number of orbits, thus when lifting to $\hat{X}\times \{0\}\subset V\times V$ there are also finitely many $G\times\Cx\times\Cx$-orbits. Therefore $\widehat{\cM}$ is a twisted $G\times\Cx\times\Cx$-equivariant coherent $D_{V^\vee\times V^\vee}$-module in the sense of \cite{Ho} whose support $\text{Supp}\,\widehat{\cM}=\hat{X}\times \{0\}$ consists of finitely many  $G\times\Cx\times\Cx$-orbits.
Thus the $\widehat{\cM}$ is regular holonomic \cite{Bo}.

The D-module $\cM=D_{V^\vee\times V^\vee}/\cJ$ is homogeneous since the ideal $\cJ$ is generated by homogeneous elements under the graduation $\deg \frac{\partial}{\partial a_i}=\deg \frac{\partial}{\partial b_i}=-1$ and $\deg a_i=\deg b_i=1$. Thus $\cM$ is regular holonomic since its Fourier transform $\widehat{\cM}$ is regular holonomic \cite{Br}.
\end{proof}

\section{Scalar systems for higher derivatives}

Now we take derivative of system \eqref{devec} and get a new scalar valued system whose solution consists of $n^2$ functions $\phi_{lk}:=\frac{\partial}{\partial a_l}\frac{\partial}{\partial a_k}\phi(a)$, $1\leq l,k \leq n$.

\begin{subequations}\label{de2vec}
    \begin{align}[left = \empheqlbrace\,]
      & P(\partial_a)\phi_{lk}(a)=0,\;\forall 1\leq k,l\leq n\label{de2vec1}\\
      & %\sum_{i,j} x_{ij}a_if_j(a)\label{de2vec2}=0\\
      (\sum_{i,j} x_{ij}a_i\frac{\partial}{\partial a_j})\phi_{lk}(a)+\sum_j x_{lj}\phi_{jk}(a)+\sum_j x_{kj}\phi_{lj}(a)=0,\;\forall   k,l\label{de2vec2}\\
      &(\sum_i a_i\frac{\partial}{\partial a_i}+3)\phi_{lk}(a)=0,\;\forall 1\leq k,l\leq n\label{de2vec3}\\
      &\phi_{lk}(a)=\phi_{kl}(a),\;\forall 1\leq k,l\leq n\label{de2vec4}\\
      & \frac{\partial}{\partial a_i}\phi_{jk}(a)=\frac{\partial}{\partial a_j}\phi_{ki}(a)=\frac{\partial}{\partial a_k}\phi_{ij}(a),\;\forall 1\leq i,j,k\leq n\label{de2vec5}.
    \end{align}
\end{subequations}

And considering $\phi(a,b):=\sum_{l,k} b_lb_k\phi_{l,k}(a)$, we get a new system:

\begin{subequations}\label{de2sca}
    \begin{align}[left = \empheqlbrace\,]
      & P(\partial_a)\phi(a,b)=0\quad \label{de2sca1}\\
      & (\sum_{i,j} x_{ij}a_i\frac{\partial}{\partial a_j}+\sum_{i,j} x_{ij}b_i\frac{\partial}{\partial b_j})\phi(a,b)=0\label{de2sca2}\\
      &(\sum_i a_i\frac{\partial}{\partial a_i}+3)\phi(a,b)=0\quad \label{de2sca3}\\
      &(\sum_i b_i\frac{\partial}{\partial b_i}-2)\phi(a,b)=0\label{de2sca4}\\
      &\frac{\partial}{\partial b_i}\frac{\partial}{\partial b_j}\frac{\partial}{\partial b_k}\phi(a,b)=0,\;\forall 1\leq i,j,k\leq n\label{de2sca5}\\
      &\frac{\partial}{\partial a_i}\frac{\partial}{\partial b_j}\frac{\partial}{\partial b_k}\phi(a,b)=\frac{\partial}{\partial a_j}\frac{\partial}{\partial b_k}\frac{\partial}{\partial b_i}\phi(a,b)=\frac{\partial}{\partial a_k}\frac{\partial}{\partial b_i}\frac{\partial}{\partial b_j}\phi(a,b),\forall i,j,k\label{de2sca6}.
      \end{align}
\end{subequations}
Similar to the previous case, we have:
\begin{prop}
By setting $\phi(a,b)=\sum_{l,k} b_lb_k\phi_{l,k}(a)$ and $\phi_{lk}(a)=\frac{\partial}{\partial b_l}\frac{\partial}{\partial b_k}\phi(a,b)$, the systems \eqref{de2vec} and \eqref{de2sca} are equivalent.
\end{prop}
\begin{proof}
Here we check for \eqref{de2sca2} and the rest is clear.
Let $\phi(a,b)=\sum_{l,k} b_lb_k\phi_{l,k}(a)$, then  the left-hand side of \eqref{de2sca2} becomes 
\begin{align*}
&\sum_{i,j,k,l}b_kb_l x_{ij }a_i\frac{\partial}{\partial a_j}\phi_{l,k}(a)+\sum_{i,j,k,l} x_{ij}\delta_{kj}b_ib_l\phi_{lk}(a)+\sum_{i,j,k,l} x_{ij}\delta_{kl}b_ib_k\phi_{lk}(a)\\\overset{\eqref{de2vec4}}{=}&\sum_{i,j,k,l}b_kb_l x_{ij }a_i\frac{\partial}{\partial a_j}\phi_{l,k}(a)+2\sum_{j,k,l} x_{lj}b_kb_l\phi_{jk}(a)
\end{align*}
And \eqref{de2vec2} implies that 
\begin{align*}
&-\sum_{k,l} b_kb_l  \sum_{i,j} x_{ij}a_i\frac{\partial}{\partial a_j}\phi_{lk}(a)\\=&\sum_{k,l} b_kb_l \sum_j x_{lj}\phi_{jk}(a)+\sum_{k,l} b_kb_l\sum_j x_{kj}\phi_{lj}(a)\\\overset{\eqref{de2vec4}}{=}&2\sum_{j,k,l}  b_kb_lx_{lj} \phi_{jk}(a)
\end{align*}
Therefore   the left-hand side of \eqref{de2sca2} equals 0. The reverse direction is also clear.
\end{proof}
\begin{prop}
Assume that the $G$-variety $X$ has only a finite number of $G$-orbits, then system \eqref{de2sca} is regular holonomic.
\end{prop}

The proof of Theorem \ref{holonomicity} follows here.
%Since the Fourier transform of the operator induced by \eqref{de2sca5} implies that $\xi_i=0$ for all $i$ and then $(V\times V)/I\subset V\times \{0\}$, the proof of Theorem \ref{holonomicity} works here as well.

\bs
\bs

In general, for $p$-th derivatives of $\phi(a)$ satisfying $\tau$, we can construct a new system as follows:

\begin{subequations}\label{de3sca}
    \begin{align}[left = \empheqlbrace\,]
      & P(\partial_a)\phi(a,b)=0\quad \label{de3sca1}\\
      & (\sum_{i,j} x_{ij}a_i\frac{\partial}{\partial a_j}+\sum_{i,j} x_{ij}b_i\frac{\partial}{\partial b_j})\phi(a,b)=0\label{de3sca2}\\
      &(\sum_i a_i\frac{\partial}{\partial a_i}+1+p)\phi(a,b)=0\quad \label{de3sca3}\\
      &(\sum_i b_i\frac{\partial}{\partial b_i}-p)\phi(a,b)=0\label{de3sca4}\\
      &\frac{\partial}{\partial b_{k_1}}\cdots\frac{\partial}{\partial b_{k_p}}\phi(a,b)=0\quad\forall 1 \leq k_i\leq n\label{de3sca5}\\
      &\frac{\partial}{\partial a_{k_1}}\frac{\partial}{\partial b_{k_2}}\cdots\frac{\partial}{\partial b_{k_p}}\phi(a,b)=\frac{\partial}{\partial a_{k_{l_1}}}\frac{\partial}{\partial b_{k_{l_2}}}\cdots\frac{\partial}{\partial b_{k_{l_p}}}\phi(a,b)\label{de3sca6}%=\frac{\partial}{\partial a_k}\frac{\partial}{\partial b_i}\cdots\frac{\partial}{\partial b_j}      
      \end{align}
\end{subequations}
where $1\leq k_i \leq n$ and $\{l_1,    \ldots, l_p\}$ is any permutation of $\{1,\ldots,p\}$.

\bs

The relationship between $\phi(a,b)$ and the derivatives of $\phi(a)$ is:
\[\phi(a,b)=\sum_{k_1,\ldots,k_p}b_{k_1}\cdots b_{k_p}\frac{\partial}{\partial a_{k_1}}\cdots\frac{\partial}{\partial a_{k_p}}\phi(a).\]

From our previous argument it is clear that if $G$ acts on $X$ with finitely many orbits, this system is regular holonomic.

\bs
By Theorem \ref{tautological}, $\sum_{k_1,\ldots,k_p}b_{k_1}\cdots b_{k_p}\frac{\partial}{\partial a_{k_1}}\cdots\frac{\partial}{\partial a_{k_p}}\Pi_\gamma(a)$ are solutions to system \eqref{de3sca}.

\red{Where you able to construct the equivalent scalar system for general $p$?} \blue{It is for $p$.}

\section{Differential relations}
\blue{Need some introduction here?}
\begin{thm}\cite{BHLSY},\cite{HLZ}\label{isom}
There are isomorphisms between following D-modules:
\begin{align*}
\tau(G,X,\omega_X^{-1},\beta_0)\leftrightarrow R[a]e^f/&Z^*%_{f,\beta_0}^*
(\hat{\gg})(R[a]e^f)\leftrightarrow 
 (\cO_{V^\vee}\boxtimes \omega_X)|_U\otimes_{\fg}\C\\
%\cD_{V^\vee}[\frac{\Omega}{f}]\\
%\partial_i, a_i\leftrightarrow\quad\partial_i&,a_i\quad\leftrightarrow \nabla_i,a_i\\
1\quad\longleftrightarrow \quad\quad\quad&e^f \quad\quad\quad\longleftrightarrow\quad \quad\quad\frac{\Omega}{f}
\end{align*}
where $R=\C[V]/I(\hat{X})$ and $R[a]=R[V^\vee].$
\end{thm}

%\begin{cor}
%The map $$\tau\ra (\cO_{V^\vee}\boxtimes \omega_X)|_U\otimes_{\fg}\C, \quad1\mapsto \frac{\Omega}{f}$$ is surjective. %Moreover, if $X$ is a projective homogeneous space, then this is an isomorphism.
%\end{cor}

%Let $X$ be a homogeneous space? No.
 These D-module isomorphisms allow us to extract some explicit information regarding vanishing of periods and their derivatives as follows.

\green{Perhaps we use $\tau$ instead of $\tau_f$ for the D-module.} \blue{OK}

\green{Some brief introduction seems needed here: these D-module isomorphisms allow us to extract some explicit information regarding vanishing of periods and its derivatives.} \blue{Fixed, is it OK?}

\bs
The isomorphisms in the Theorem \ref{isom} induce, for each $f_a\in B$, the isomorphism \cite[Theorem 2.9]{BHLSY}
\[\Psi: (Re^{f_a}/Z^*(\hat{\gg})(Re^{f_a}))^*\xrightarrow{\simeq} \Hom_{D_{V^\vee}}(\tau,\cO_a).\]
In particular, if $\Gamma$ is a $d$-cycle in $X-V(f_a)$, then the linear function 
$$\lambda_\Gamma:R e^{f_a}\ra\C,\hskip.2in p(\zeta)e^{f_a}\mapsto \left.\left(p(\partial)\int_{\Gamma}\frac{\Omega}{f}\right)\right|_{f_a}$$
vanishes on the subspace $Z^*(\hat{\gg})(Re^{f_a})$.
%sending $e^{f_a}$ to $\int_{\Gamma}\frac{\Omega}{f_a}$ and $p(\zeta)e^{f_a}$ to $p(\partial)\int_{\Gamma}\frac{\Omega}{f_a}$ for $p(\zeta)\in R$ and $\Gamma$ a $d$-tube cycle over the $(d-1)$-cycle $\gamma$ in $X$.

\begin{cor}\label{relation}
Given $p(\zeta)\in R$,
define 
$$\cN(p):=\{f_a\in B\mid p(\zeta)e^{f_{a}}\in Z^*(\hat{\gg})(Re^{f_{a}})\}.$$ 
Then it is a closed subset of $B$. Moreover, for any $f_a\in\cN(p)$ we have a differential relation: $$\left.\left(p(\partial)\int_{\Gamma}\frac{\Omega}{f}\right)\right|_{f_a}=0.$$\end{cor}
%Note that the relation we get is not a differential equation, since it is not true for all $a\in B$.
%\end{cor}

\bs
If $X=\P^d$, then $\omega_X^{-1}=\cO(d+1)$ and $V=\Gamma(X,\cO(d+1))^\vee$. We can identify $R$ with the subring of $\C[x]:=\C[x_0,
\ldots,x_d]$ consisting of polynomials spanned by monomials of degree divisible by $d+1$.
 %$f_a=x_0^{n+1}+\cdots+x_n^{n+1}$ is the Fermat polynomial.
By Lemma 2.12 in \cite{BHLSY}, 
%It is also shown in \cite{BHLSY} that 
\[\hat{\gg}\cdot(Re^f)=Re^f\cap\sum_i\frac{\partial}{\partial x_i}(\C[x]e^f),\]
meaning that elements of the form $\frac{\partial}{\partial x_i} x_jp(x)e^f$ are in $Z^*(\hat{\gg})(Re^{f})$ for $p(x)\in R$.

Now we specialize to the Fermat case $f_F=x_0^{d+1}+\cdots+x_d^{d+1}$ and look at two examples.
\begin{exm}
Let $X=\P^1$, $f=a_0x_0x_1+a_1x_0^2+a_2x_1^2$. We look at $\frac{\partial}{\partial a_0}e^{f_F}=x_0x_1e^{f_F} $. %at the Fermat point $f_a=x_0^2+x_1^2$. 
Since $(\frac{\partial}{\partial x_0}x_1)e^{f_F}=2x_0x_1e^{f_F}\in Z^*(\hat{\gg})(Re^{f_F}),$ %in $R[a]e^f&/Z^*R[a]e^f$?
it implies that $f_F\in \cN(x_0x_1)$%$\frac{\partial}{\partial a_0}e^{f_F}=0$ %then $\frac{\partial}{\partial a_0}\frac{\Omega}{{f_F}}=0$ 
and by Corollary \ref{relation} we have a differential relation $\left.\left(\frac{\partial}{\partial a_0}\int_{\Gamma}\frac{\Omega}{f}\right)\right|_{f_F}=0$. 

We know that when $|a_0|>>0$, $\int_{\Gamma}\frac{\Omega}{f_F}=\frac{1}{a_0}\sum_{k=0}^{\infty}\frac{(2k)!}{k!k!}(\frac{a_1a_2}{a_0^2})^k$, so the above equality means that any analytic continuation of this power series function must have a vanishing derivative with respect to $a_0$ at the Fermat point, which we can verify easily since the analytic continuation is given explicitly by $ (a_0^2-4a_1a_2)^{-\frac{1}{2}}$.
\end{exm}
\bs

\begin{exm}
Let $X=\P^2$, $f=a_0x_0x_1x_2+a_1x^2_0x_1+a_2x_0x^2_1+a_3x^3_1+a_4x^2_1x_2+a_5x_1x_2^2+a_6x^3_2+a_7x_0x^2_2+a_8x^2_0x_2+a_9x_0^3$, and $\frac{\partial^2}{\partial a_0^2}e^{f_F}=(x_0x_1x_2)^2e^{f_F}.$ Since $(\frac{\partial}{\partial x_0}x_1)x_1x_2^2e^{f_F}=3x_0^2x_1^2x_2^2e^{f_F}\in Z^*(\hat{\gg})(Re^{f_F}),$ $f_F\in \cN((x_0x_1x_2)^2)$, therefore
by Corollary \ref{relation} we have a differential relation $\left.\left(\frac{\partial^2}{\partial a_0^2}\int_{\Gamma}\frac{\Omega}{f}\right)\right|_{f_F}=0$. 
\end{exm}

\section{Concluding remarks}

We conclude this paper with some remarks about differential zeros of period integrals of differential systems for period mappings.

%\begin{rem}
%It will be interesting to use this method to systematically study the zero locus of derivatives of periods.
For a given $a\in B$, what can we say about the the function space $\{p(\zeta)\in R\mid p(\zeta)e^{f_{a}}\in Z^*(\hat{\gg})(Re^{f_{a}})\}$? 
This space is unfortunately neither an ideal nor a $\hat\gg$-submodule of $R$ in general, unless $f_a$ is $\hat\gg$-invariant in which case it is equal to $Z^*(\hat{\gg})R$. %This function space tell us some global information about the period sheaf of $\cY$, i.e. it gives a set of differential zeros of the period sheaf at the point $a$. 
However, it seems that it is more interesting to consider the closed subset $\cN(p)\subset B$, for each function $p(\zeta)\in R$ that we defined above. For this is an algebraic set that gives us the vanishing locus of certain derivatives (corresponding to $p$) of all period integrals. For example, it is well known that period integrals of certain CY hypersurfaces can be represented by (generalized) hypergeometric functions. In this case, the vanishing locus above therefore translates into a monodromy invariant statement about differential zeros of the hypergeometric functions in question.
%\end{rem}

A remark about our new differential systems for period mappings is in order.
For the family of CYs $\cY$, since the period mapping is given by higher derivatives of the periods of $(d-1,0)$ forms, any information about the period mapping can in principle be derived from period integrals, albeit somewhat indirectly. However, the point here is that an explicit regular holonomic system for the full period mapping would give us a way to study the structure of this mapping by D-module techniques directly. Thanks to the Riemann-Hilbert correspondence, these techniqus have proven to be a very fruitful approach to geometric questions about the family $\cY$ (e.g. degenerations, monodromy, differential zeros, etc) \cite{BHLSY,HLZ} when applied to $\tau$. The new differential system that we have constructed for the period mapping is in fact nothing but a tautological system. Namely, it is a regular holonomic D-module defined by a polynomial ideal together with a set of first order symmetry operators -- conceptually of the same type as $\tau$.  It is therefore {\it directly} amenable to the same tools (Fourier transforms, Riemann-Hilbert, Lie algebra homology, etc) we applied to investigate $\tau$ itself. The hope is that understanding the structure of the new D-module will shed new light on Hodge-theoretic questions about the family $\cY$. For example, there is an analogue Theorem \ref{isom} which allows us to construct differential zeros of solution sheaves for this class of D-modules. It would be interesting to understand their implications about period mappings. We would like to return to these questions in a future paper.

\section{Acknowledgement}

We dedicate this paper to our mentor, teacher and colleague Professor Shing-Tung Yau on the occasion of his 65th birthday, 60 years after the celebrated Calabi conjecture. The authors AH and BHL also thank their collaborators S. Bloch, R. Song, D. Srinivas, S.-T. Yau and X. Zhu for their contributions to the foundation of the theory of tautological systems, upon which this paper is based. Finally, we thank the referee for corrections and helpful comments. BHL is partially supported by NSF FRG grant DMS 1159049.

%\newpage

\vskip.1in

\noindent\address {\SMALL J. Chen, Department of Mathematics, Brandeis University, Waltham MA 02454.\\ jychen@brandeis.edu.}
%\vskip-.15in

\noindent\address {\SMALL A. Huang, Department of Mathematics, Harvard University, Cambridge MA 02138. \\ anhuang@math.harvard.edu.}
%\vskip-.15in

\noindent\address {\SMALL B.H. Lian, Department of Mathematics, Brandeis University, Waltham MA 02454.\\ lian@brandeis.edu.}
%\vskip-.15in


\begin{thebibliography}{99}





\bibitem[BHLSY]{BHLSY} S. Bloch, A. Huang, B.H. Lian, V. Srinivas, S.-T. Yau, {\it On the Holonomic Rank Problem}, J. Differential Geom. 97 (2014), no. 1, 11-35.

\bibitem[Bo]{Bo} A. Borel et al, {\it Algebraic D-modules}, Perspectives in Mathematics, 2. Academic Press, Inc., Boston, MA, 1987. 

\bibitem[Br]{Br} J.-L. Brylinski, {\it Transformations canoniques, dualit\'e projective, th\'eorie de Lefschetz, transformations de Fourier et sommes trigonom\'etriques}, Ast\'erisque No.
140-141 (1986), 3134, 251.


\bibitem[G]{G} P. Griffiths, {\it The residue calculus and some transcendental results in algebraic geometry, I}. Proc. Natl. Acad. Sci. USA, 55 (5) (1966), 1303-1309.

%\bibitem[GH]{GH} P. Griffiths and J. Harris, {\it Principles of Algebraic Geometry}, John Wiley and Sons, Inc., 1994.


\bibitem[HLZ]{HLZ} A. Huang, B.H. Lian, X. Zhu, {\it Period Integrals and the Riemann-Hilbert Correspondence}, 	arXiv:1303.2560v2.

\bibitem[Ho]{Ho} R. Hotta, {\it Equivariant D-modules}, arXiv:math/9805021v1.

\bibitem[LSY]{LSY}  B.H. Lian, R. Song and S.-T. Yau, {\it Periodic Integrals and Tautological Systems},  J. Eur. Math. Soc. (JEMS) 15 (2013), no. 4, 1457-1483.

\bibitem[LY]{LY}  B.H. Lian and S.-T. Yau, {\it Period Integrals of CY and General Type Complete Intersections}, Invent. Math. 191 (2013), no. 1, 35-89, 2013.


%\bibitem[Ka]{Ka} M. Kapranov, {\it Hypergeometric functions on reductive groups}, Integrable systems and algebraic geometry (Kobe/Kyoto, 1997), 236-281, World Sci. Publ., River Edge, NJ, 1998.  



\bibitem[Vo]{Vo} C. Voisin, {\it Hodge Theory and Complex Algebraic Geometry}, Cambridge University Press, Vol. I \& II, 2002.


\end{thebibliography}
\end{document}